%% file: Regular_Equipments.tex
\documentclass[12pt,oneside,article]{memoir}
\input{preamble.tex}
\settrims{0pt}{0pt} 
\setlxvchars 
\settypeblocksize{*}{36pc}{*} 
\setlrmargins{*}{*}{1} 
\setulmarginsandblock{1in}{1in}{*} 
\setheadfoot{\onelineskip}{2\onelineskip} 
\setheaderspaces{*}{1.5\onelineskip}{*} 
\checkandfixthelayout

\setsecheadstyle{\bfseries\large\raggedright}
\setsubsecheadstyle{\bfseries\raggedright}

\newcommand{\fig}[1]{\vcenter{\hbox{\includegraphics{figures/fig#1}}}}

\title{Regular and exact (virtual) double categories}
\author{Patrick Schultz\thanks{The author was supported by the following grants: Office of Naval
Research ONR N00014131 0260, Air Force Office of Scientific Research AFOSR FA9550--14--1--0031, and
National Aeronautics and Space Administration NASA (Langley Research Center) NNH13ZEA001N--SSAT.}}
\date{\vspace{-3ex}}

\begin{document}
\tightlists
\firmlists

\maketitle

\begin{abstract}
   We propose definitions of regular and exact (virtual) double categories, proving a number of
   results which parallel many basic results in the theory of regular and exact categories. We show
   that any regular virtual double category admits a factorization system which generalizes the
   factorization of a functor between categories into a bijective-on-objects functor followed by a
   fully-faithful functor. Finally, we show that our definition of exact double category is
   equivalent to an axiom proposed by Wood, and very closely related to the ``tight Kleisli
   objects'' studied by Garner and Shulman.
\end{abstract}

\chapter{Introduction}

Category theory, besides having proven itself very generally useful, with examples of categories
arising in most every branch of mathematics, has also proven itself very generalizable. For
instance, enriched categories, internal categories, fibered categories, and quasicategories are some
of the many variations and generalizations of the definition of category which have established
themselves in modern mathematics.

Each of these has a theory which closely parallels that of ordinary categories, with functors and
natural transformations, adjunctions, (weighted) limits and colimits, (pointwise) Kan extensions,
the Yoneda lemma, and so on all playing central roles. It is natural to search for a common
framework in which this body of definitions and results---which we refer to as \emph{formal category
theory}---can be developed once and specialized to each existing and future collection of
``category-like structures''.

The obvious candidate for such a common framework is the theory of \emph{2-categories}, or their
less strict variation, \emph{bicategories}. Every example of ``category-like structures'' can be
assembled into a bicategory, so many people have tried to develop formal category theory at the
level of generality of an arbitrary 2-category.  However, it was quickly apparent that without more
structure, important concepts like weighted limits and colimits and the Yoneda embedding do not have
an adaquate expression.

One proposal for extra structure supporting a robust formal category theory was given by Wood
in~\cite{Wood:1982a,Wood:1985a}. The motivation for his proposal is that, besides functors and
natural transformations, \emph{profunctors} between categories are also a fundamental part of
category theory (though often in the background). Wood defined an extra stucture on a bicategory
$\ccat{B}$, together with a set of axioms,  which ``equips $\ccat{B}$ with abstract proarrows''. We
will refer to this structure as a \emph{proarrow equipment} for short.

In~\cite{Shulman:2008a}, Shulman showed that (pseudo) double categories satisfying a simple property
are essentially equivalent to Wood's proarrow equipments. Shulman called these double categories
\emph{framed bicategories}, though in~\cite{Cruttwell.Shulman:2010a} and elsewhere he has switched
to refering to them simply as \emph{equipments}, which we will do as well. He moreover demonstrated
that the double category formulation of equipments makes clear the ``right'' definitions of functors
and transformations, leading to a well-behaved 2-category of equipments.

In~\cite{Cruttwell.Shulman:2010a}, Cruttwell and Shulman generalized equipments to \emph{virtual
equipments}, which are \emph{virtual double categories} satisfying some simple properties. In a
virtual equipment, composition of proarrows may not exist, yet there is still enough structure to
support the development of formal category theory. They also show that all types of ``generalized
multicategory'', of which the majority of category-like structures are examples, arise as the
objects in the virtual equipment of ``monoids and modules`` in some virtual equipment.  Thus we can
see that, just as most known types of algebraic or geometric structure can be assembled into a
category, most known types of category-like structures (and more besides) can be assembled into a
virtual equipment.

\plainbreak{1}

In classical category theory, there is a hierarchy of additional properties a category $\cat{C}$
might have, beginning with $\cat{C}$ simply having finite limits, and culminating with $\cat{C}$
being a Grothendieck topos.  The higher up this hierarchy $\cat{C}$ is, the more ``set-like'' it is.
Some of the intermediate levels in this hierarchy are regular, exact, coherent, and extensive
categories, and pretoposes.

In this paper, we propose a beginning to an analogous hierarchy of additional properties on a
virtual equipment. The higher up this hierarchy a virtual equipment $\dcat{D}$ is, the more
properties it shares with categories and profunctors, and hence the more elements of formal category
theory it should be possible to interpret inside $\dcat{D}$. In particular, we propose in this paper
definitions of \emph{regular virtual equipment} and of \emph{exact virtual equipment}.

In Section~\ref{ch:review}, we review the definitions of (virtual) double category and (virtual)
equipment, as well as the construction of monoids and modules in a virtual double category. In
Section~\ref{ch:collapse}, we define the ``collapse'' of a monoid, which plays a role in the
theory analogous to coequalizers in the theory of regular and exact categories, and which is closely
related to the Kleisli object of a monad.

Section~\ref{ch:regular} gives the definition of regular virtual equipment and proves some basic
results which parallel the typical exposition of regular categories. In particular, we show that
just as every regular category has a factorization system generalizing the epi/mono image
factorization in $\Set$, every regular virtual equipment has a factorization system generalizing the
bijective-on-objects/fully-faithful image factorization in $\Cat$.

Lastly, Section~\ref{ch:exact} gives the definition of exact virtual equipment. The main result in
this section is that exactness in our sense is essentially equivalent to Wood's ``Axiom 5''
from~\cite{Wood:1985a}. This axiom, and the closely related ``tight Kleisli objects''
from~\cite{Garner.Shulman:2013a}, involves Kleisli objects and Eilenberg-Moore objects for monads in
the bicategory of proarrows. While those papers clearly show that this is an important construction
for formal category theory, it always felt to the author to be counter to Shulman's ``philosophy
that the [proarrows] are not `morphisms', but rather objects in their own
right''~\cite{Shulman:2008a}. The definition of exact virtual equipment gives an equivalent
condition which we feel adheres to this philosophy, and establishes a tight analogy with a large
body of classical category theory which we hope will stimulate further work in this direction.

\plainbreak{1}

The author would like to thank Mike Shulman for helpful conversations, as well as David Spivak for
helpful conversations and feedback on drafts of this paper.

\subsubsection{Notational conventions}

This paper deals with categories, 2-categories/bicategories, and (virtual) double categories, and so
it is helpful to establish a notational convention to keep straight the various structures. In this
paper, we write category variables $\cat{C}$ in a caligraphic font (except when working inside the
equipment $\dProf$, where it would be distracting), while we write named categories such as $\Set$
and $\Cat$ in a bold roman font. 2-categories and bicategories such as $\CCat$ we write with a
script-style first letter, and bicategory variables $\ccat{B}$ similarly. Double categories and
virtual double categories we write with the first letter in a blackboard font: $\dcat{D}$, $\dProf$.

\chapter{(Virtual) double categories and equipments}\label{ch:review}

We begin by recalling some definitions from~\cite{Shulman:2008a,Cruttwell.Shulman:2010a} which are
at the center of the present paper.

\begin{definition}\label{def:virtual_double_category}
   A \emph{virtual double category} $\dcat{D}$ consists of the following data:
   \begin{compactitem}
      \item A category $\dcat{D}_0$, which we refer to as the \emph{vertical category} of
         $\dcat{D}$. For any two objects $c,d\in\dcat{D}_0$, we will write
         $\dcat{D}(c,d)=\dcat{D}_0(c,d)$ for the set of vertical arrows from $c$ to $d$.
      \item For any two objects $c,d\in\dcat{D}_0$, a set of horizontal arrows, which we refer to as
         proarrows and draw with a slash: $c\tickar d$.
      \item 2-cells, which have the shape
         \begin{equation}\label{eq:virtual_2-cell}
            \fig{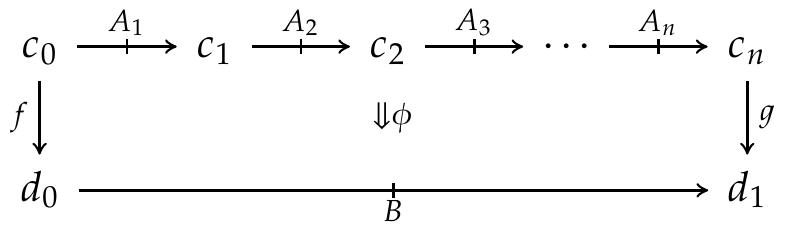}
         \end{equation}
         for any $n\geq 0$. We will call $f$ and $g$ the \emph{left frame} and \emph{right frame} of
         $\phi$, and call the string $A_1,\dots,A_n$ the (multi-)source and $B$ the target of
         $\phi$. We will write ${}_f\dcat{D}_g(A_1,\dots,A_n;B)$ for the set of all cells of
         shape~\eqref{eq:virtual_2-cell} in $\dcat{D}$, and we write $\dcat{D}(A_1,\dots,A_n;B)$ for
         the set of cells with $f$ and $g$ identities.
      \item For each proarrow $A\colon c\tickar d$ there is an identity 2-cell
         \begin{equation*}
            \fig{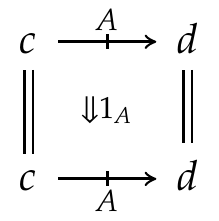}
         \end{equation*}
      \item Composition of 2-cells is like composition in a multicategory. So given the 2-cell
         $\phi$ in~\eqref{eq:virtual_2-cell} and $n$ other 2-cells with horizontal targets
         $A_1,\dots,A_n$, there is a composite 2-cell with the evident shape. This composition
         operation satisfies unit and associativity axioms like in a multicategory.
   \end{compactitem}
\end{definition}

We will now introduce the primary running examples of this paper.

\begin{example}
   There is a virtual double category $\dRel$ with vertical category $\dRel_0=\Set$, and with
   proarrows $R\colon a\tickar b$ given by relations $R\subseteq b\times a$. There is a 2-cell of
   the form~\eqref{eq:virtual_2-cell} if and only if for every tuple $(x_0,\dots,x_n)\in
   c_0\times\dots\times c_n$, the implication
   \begin{equation*}
      A_1(x_1,x_0)\wedge\dots\wedge A_n(x_n,x_{n-1}) \Rightarrow B(g(x_n),f(x_0)).
   \end{equation*}
   holds.
\end{example}

\begin{example}
   There is a virtual double category $\dProf$ with vertical category $\dProf_0=\Cat$, and with
   proarrows $P\colon C\tickar D$ given by profunctors $P\colon\op{D}\times
   C\to \Set$. Given an element $x\in P(d,c)$ and morphisms $f\colon c\to c'$ in $C$ and
   $g\colon d'\to d$ in $D$, we will write the functorial action as $P(g,f)(x)=f\cdot x\cdot
   g$.

   A 2-cell of the form
   \begin{equation}
      \fig{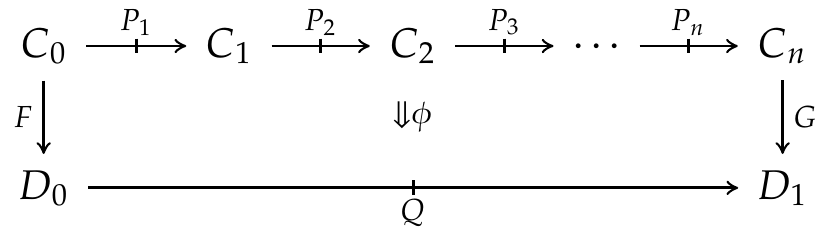}
   \end{equation}
   is a family of functions $P_1(c_1,c_0)\times\dots\times P_n(c_n,c_{n-1})\to Q(Gc_n,Fc_0)$ for
   each tuple of objects $(c_0,\dots,c_n)\in C_0\times\dots\times C_n$, which is natural in each of
   the $C_i$. For $C_0$, naturality means for each $f\colon c_0\to c'_0$ and each
   $(x_1,\dots,x_n)\in P_1(c_1,c_0)\times\dots\times P_n(c_n,c_{n-1})$, we have
   $\phi(f\cdot x_1,x_1,\dots,x_n)=F(f)\cdot\phi(x_1,\dots,x_n)$, while naturality in $C_1$ means
   for each $g\colon c_1\to c'_1$ we have $\phi(x_1\cdot g,x_2,\dots,x_n)=\phi(x_1,g\cdot
   x_2,\dots,x_n)$, and similarly for $C_2,\dots,C_n$.
\end{example}

\begin{definition}\label{def:monoids_and_modules}
   Let $\dcat{D}$ be a virtual double category. The virtual double category $\dMod(\dcat{D})$ of
   \emph{monoids and modules} is defined as follows:
   \begin{compactitem}
      \item The objects are \emph{monoids} in $\dcat{D}$: tuples $(c,M,e_M,m_M)$ consisting of an
         object $c$ of $\dcat{D}$, a proarrow $M\colon c\tickar c$, and unit and
         multiplication cells
         \begin{equation*}
            \vcenter{\hbox{\includegraphics{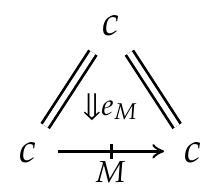}\hspace{4em}\includegraphics{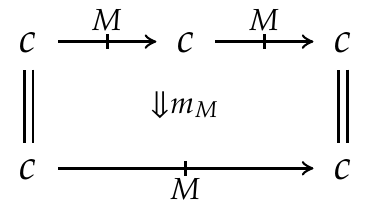}}}
         \end{equation*}
         satisfying the evident unit and associativity axioms.
      \item The vertical arrows are \emph{monoid homomorphisms}: pairs $(f,\vec{f}\,)$ of a vertical arrow
         $f\colon c\to d$ in $\dcat{D}$ and a cell
         \begin{equation*}
            \fig{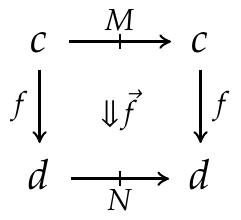}
         \end{equation*}
         which respects the unit and multiplication cells of $M$ and $N$.
      \item The proarrows $B\colon M\tickar N$ are \emph{bimodules}: triples $(B,l_B,r_B)$
         consisting of a proarrow $B\colon c\tickar d$ in $\dcat{D}$ and cells
         \begin{equation*}
            \vcenter{\hbox{\includegraphics{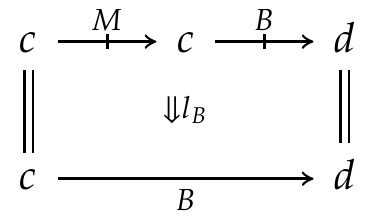}\hspace{3em}\includegraphics{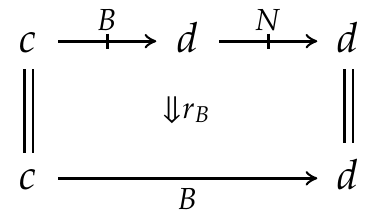}}}
         \end{equation*}
         satisfying evident monoid action axioms.
      \item The 2-cells are \emph{bimodule homomorphisms}: cells in $\dcat{D}$
         \begin{equation*}
            \fig{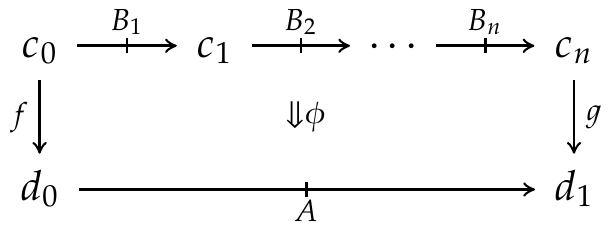}
         \end{equation*}
         which are compatible with the left and right actions of the bimodules.
   \end{compactitem}
\end{definition}

\begin{remark}
   Given two monoids $M\colon c\tickar c$ and $N\colon d\tickar d$ in a virtual double category
   $\dcat{D}$, we will write ${}_M\Bimod_N$ for the \emph{category} of $(M,N)$-bimodules,
   i.e.~proarrows $M\tickar N$ in $\dMod(\dcat{D})$.
\end{remark}

In a multicategory, tensor products of objects can be captured via a universal property. In this
way, monoidal categories are equivalent to multicategories in which the tensor product of any list
of objects exists. Similarly, composition of proarrows in a virtual double category can be captured
by a universal property, and virtual double categories in which all such composites exist are
equivalent to double categories. Note that in this paper, as in \cite{Shulman:2008a}, double
categories are always assumed to be \emph{pseudo double categories}, in which composition in the
vertical direction is strictly associative and unital, and in which composition in the horizontal
direction is associative and unital only up to coherent isomorphism.

\begin{definition}
   A cell
   \begin{equation}\label{eq:opcart_cell}
      \fig{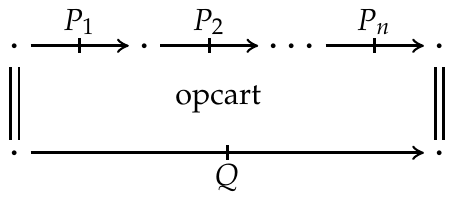}
   \end{equation}
   in a virtual double category is said to be \emph{opcartesian} if any cell
   \begin{equation*}
      \fig{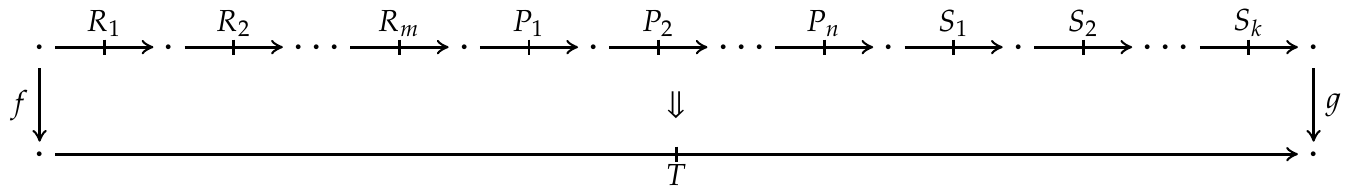}
   \end{equation*}
   factors through it uniquely as
   \begin{equation*}
      \fig{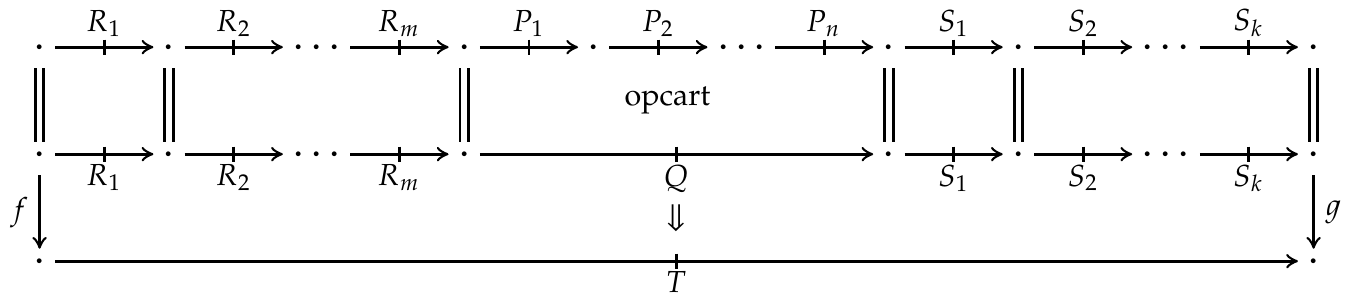}
   \end{equation*}
   Thus a cell of the form~\eqref{eq:opcart_cell} is opcartesian precisely if composition with it
   induces a bijection
   \begin{equation*}
      {}_f\dcat{D}_g(R_1,\dots,R_m,Q,S_1,\dots,S_k;T)\iso{}_f\dcat{D}_g(R_1,\dots,R_m,P_1,\dots,P_n,S_1,\dots,S_k;T)
   \end{equation*}
   for any $f$, $g$, $T$, $R_1,\dots,R_m$, and $S_1,\dots,S_k$.

   Whenever an opcartesian cell~\eqref{eq:opcart_cell} exists, we will refer to $Q$ as \emph{the
   composite} of the $P_i$'s, and write it as $P_1\odot\cdots\odot P_n$. In the $n=0$ case, if there
   is an opcartesian cell of the form
   \begin{equation*}
      \fig{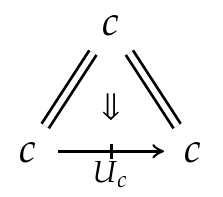}
   \end{equation*}
   we say that $c$ \emph{has a unit} $U_c$. When clear from context, we will often write $c$ for the
   unit proarrow $U_c$. Likewise, for any vertical arrow $f\colon c\to d$ we will often write $f$
   for the unit 2-cell
   \begin{equation*}
      \fig{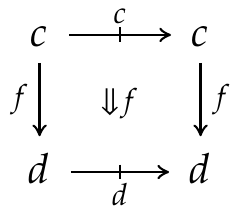}
   \end{equation*}
   which is induced by $f$ using the universal property of the units.
\end{definition}

\begin{definition}
   Say that a virtual double category $\dcat{D}$ \emph{has units} if every object has a unit. Say
   that $\dcat{D}$ \emph{has composites} if every string of $n\geq 0$ composable proarrows has a composite.
\end{definition}

\begin{definition}
   If a virtual double category $\dcat{D}$ has units, then we can define a \emph{vertical 2-category}
   $\VVer(\dcat{D})$. The objects and morphisms of $\VVer(\dcat{D})$ are the objects and vertical
   arrows of $\dcat{D}$, while for any pair of morphisms $f,g\colon c\to d$, the 2-cells $\phi\colon
   f\Rightarrow g$ are defined to be 2-cells in $\dcat{D}$ of the form
   \begin{equation*}
      \fig{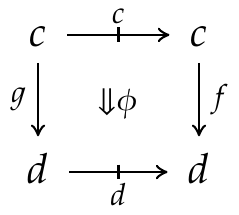}
   \end{equation*}

   If $\dcat{D}$ has \emph{all} composites, then we can also define a \emph{horizontal bicategory}
   $\HHor(\dcat{D})$. The objects and morphisms of $\HHor(\dcat{D})$ are the objects and proarrows
   of $\dcat{D}$, and the 2-cells are the 2-cells of $\dcat{D}$ with identity left and right frames.
   The bicategory axioms follow from the universal property of the composites.

   Even if $\dcat{D}$ does not have all composites, we will sometimes abuse notation by writing
   $\HHor(\dcat{D})(c,d)$ for the category of proarrows $c\tickar d$.
\end{definition}

\begin{example}
   The virtual double category $\dRel$ has composites. For any set $A$, the unit relation $A\colon
   A\tickar A$ is simply the equality relation: $A(a_1,a_2)\Leftrightarrow a_1=a_2$. For any
   composable pair of relations $R\colon A\tickar B$, $S\colon B\tickar C$, the composite is the
   usual composition of relations:
   \begin{equation*}
      (R\odot S)(c,a) \Leftrightarrow \exists b\in B.\; R(b,a)\wedge S(c,b)
   \end{equation*}
\end{example}

\begin{example}
   The virtual double category $\dProf$ has composities as well. For any category $C$, the
   unit profunctor $C\colon C\tickar C$ is the hom profunctor $\op{C}\times C\to\Set$, thus
   $C(c_1,c_2)=\Hom_C(c_1,c_2)$. For any composable pair of profunctors $P\colon C\tickar D$ and
   $Q\colon D\tickar E$, the composite can be defined as a coend
   \begin{equation*}
      (P\odot Q)(e,c) = \int^{d\in D} P(d,c)\times Q(e,d).
   \end{equation*}
   This coend can be equivalently constructed as a quotient of $\Pi_{d\in D} (P(d,c)\times Q(e,d))$,
   where for any $f\colon d\to d'$ in $D$ and any $p\in P(d',c)$ and $q\in Q(e,d)$, we identity
   $(p\cdot f,q)$ and $(p,f\cdot q)$. In this way, profunctor composition can be seen as analogous
   to the tensor product of bimodules. This analogy between categories/profunctors and
   rings/bimodules is a very fruitful one, and in fact by generalizing to enriched categories, rings
   and bimodules can be seen as a special case of enriched categories and profunctors.
\end{example}

\begin{example}
   For any virtual double category $\dcat{D}$, the virtual double category $\dMod(\dcat{D})$ will
   always have units, though does not have all composites in general. For any monoid $(c,M)$, it is
   not hard to see that the unit bimodule is simply $M\colon c\tickar c$ regarded as a
   $(M,M)$-bimodule. In~\cite{Shulman:2008a} it is shown that if $\dcat{D}$ has composites, and has
   local reflexive coequalizers which are preserved under composition, then $\dMod(\dcat{D})$ has
   composites.
\end{example}

\begin{definition}
   A cell
   \begin{equation}\label{eq:cart_cell}
      \fig{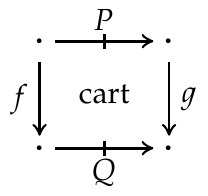}
   \end{equation}
   in a virtual double category is said to be \emph{cartesian} if any cell
   \begin{equation*}
      \fig{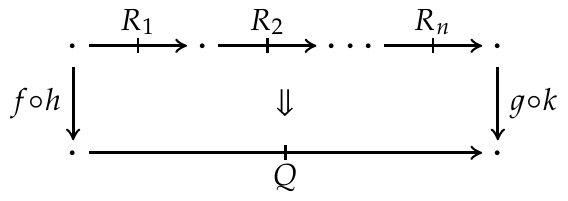}
   \end{equation*}
   factors through it uniquely as
   \begin{equation*}
      \fig{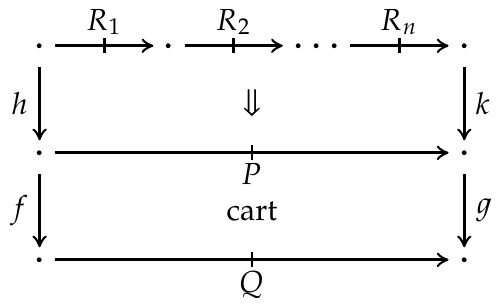}
   \end{equation*}
   Thus a cell of the form~\eqref{eq:cart_cell} is cartesian precisely if composition with it
   induces a bijection
   \begin{equation*}
      {}_h\dcat{D}_k(R_1,\dots,R_n;P)\iso{}_{fh}\dcat{D}_{gk}(R_1,\dots,R_n;Q)
   \end{equation*}
   for any $h$, $k$, and $R_1,\dots,R_n$.

   When a cartesian cell of the form~\eqref{eq:cart_cell} exists, we say that $P$ is (isomorphic to)
   the \emph{restriction} of $Q$ along $f$ and $g$, written $Q(g,f)$. We say that a virtual double
   category \emph{has restrictions} if $Q(g,f)$ exists for all compatible $Q$, $f$, and $g$.
\end{definition}

\begin{definition}
   A \emph{virtual equipment} $\dcat{D}$ is a virtual double category which has units and
   restrictions. If $\dcat{D}$ has all composites, hence is a double category, we will call
   $\dcat{D}$ an \emph{equipment} (called a \emph{framed bicategory} in~\cite{Shulman:2008a}).
\end{definition}

\begin{example}
   $\dRel$ is an equipment: given functions $f\colon A\to B$ and $g\colon C\to D$, and a relation
   $R\colon B\tickar D$, the restriction is given by $R(g,f)(c,a)\Leftrightarrow R(g(c),f(a))$.

   $\dProf$ is also an equipment: given functors $F\colon A\to B$ and $G\colon C\to D$, and a
   profunctor $P\colon B\tickar D$, the restriction is given by $P(G,F)(c,a)=P(Gc,Fa)$. In other
   words, $P(G,F)$ is the composition
   \begin{equation*}
      \fig{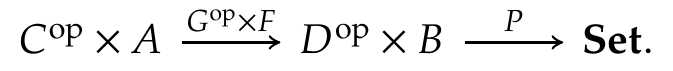}
   \end{equation*}

   If $\dcat{D}$ is a virtual equipment, then so is $\dMod(\dcat{D})$.
   See~\cite{Cruttwell.Shulman:2010a} for details.
\end{example}

\begin{example}\label{ex:Prof_is_ModSpan}
   Let $\cat{C}$ be a category with pullbacks. There is an equipment $\dSpan(\cat{C})$ whose
   vertical category is $\cat{C}$, and whose proarrows $S\colon c\tickar d$ are spans $d\leftarrow
   S\rightarrow c$. Composition of spans is formed by pullback, and the 2-cells are the evident
   thing.

   In~\cite{Shulman:2008a,Cruttwell.Shulman:2010a} it is shown that $\dMod(\dSpan(\cat{C}))$ is the
   equipment of categories, functors, and profunctors \emph{internal} to $\cat{C}$. In particular,
   $\dProf=\dMod(\dSpan(\Set))$.
\end{example}

\begin{remark}
   Any vertical arrow $f\colon c\to d$ in a virtual equipment gives rise to the two proarrows
   $d(1,f)\colon c\tickar d$ and $d(f,1)\colon d\tickar c$, formed by restricting the unit proarrow
   on $d$ along $f$ on one side and an identity on the other. We will call proarrows of this form
   \emph{representable}.

   Representable proarrows play a special role in the theory. For instance, in~\cite{Shulman:2008a}
   it is shown that if a double category has restrictions of this special form, then it in fact has
   \emph{all} restrictions.  The same is not true for virtual double categories, but the following
   proposition shows that, assuming all restrictions exist, then all restrictions can be recovered
   by composition with representable proarrows. For this reason, representable proarrows are also
   often called \emph{base change objects}.
\end{remark}

\begin{example}
   In $\dProf$, a profunctor $P\colon 1\tickar C$ is precisely a presheaf on $C$, while a functor
   $x\colon 1\to C$ is just an object of $C$. In this case, $P$ is representable by the functor $x$
   if $P\iso C(1,x)$, i.e.~if for every object $y\in C$, $P(y)\iso C(y,x)$. This is the motivation
   for the term \emph{representable profunctor}.
\end{example}

\begin{proposition}\label{prop:equipment_2-cell_bijection}
   Let $P\colon c\tickar d$ be a proarrow and $f\colon a\to c$ and $g\colon b\to d$ be vertical
   arrows in a virtual equipment. Then the composite $C(1,f)\odot P\odot B(g,1)$ exists and is
   isomorphic to $P(g,f)$.

   Moreover, there is a bijection between cells of the form
   \begin{equation*}
      \vcenter{\hbox{\includegraphics{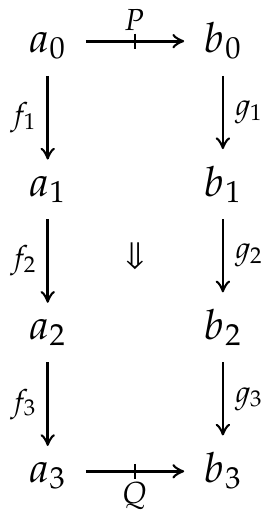}}}
      \vcenter{\hbox{\hspace{2em}and\hspace{2em}}}
      \vcenter{\hbox{\includegraphics{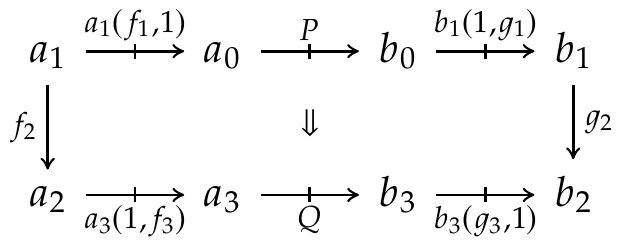}}}
   \end{equation*}
\end{proposition}
Note that we can make sense of the cell on the right because the composition of the proarrows
along the bottom exists (and is isomorphic to $Q(g_3,f_3)$). We draw it this way to make the
symmetry clear.

\chapter{Collapse}\label{ch:collapse}

In this section we will introduce a central concept of this paper: the collapse of a monoid or
bimodule in a virtual equipment. This can be seen as a generalization both of the Kleisli object of
a monad in a bicategory, and of the quotient of a relation in a category. It is essentially the same
as the ``tight Kleisli objects'' considered in~\cite{Garner.Shulman:2013a}, though they worked in a
slightly more general context.

\begin{definition}\label{def:embedding}
   An \emph{embedding} of a monoid $M\colon c\tickar c$ in a virtual equipment into an object $x$ is
   a monoid homomorphism (Definition~\ref{def:monoids_and_modules}) $(f,\vec{f}\,)$ from $M$ to the
   trivial monoid on $x$:
   \begin{equation*}
      \fig{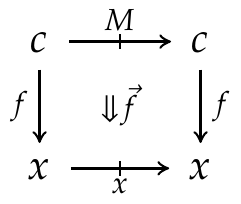}
   \end{equation*}
   We will sometimes write an embedding as $(f,\vec{f}\,)\colon(c,M)\to x$, or even just $f\colon M\to
   x$ when clear from context. We will write $\Emb(M,x)$ for the set of embeddings from $M$ to $x$.

   Likewise, an \emph{embedding} of a $(M,N)$-bimodule $B$ into a proarrow $P\colon x\tickar y$
   consists of monoid embeddings $f\colon M\to x$ and $g\colon N\to y$, and a bimodule homomorphism
   from $B$ to $P$, regarding $P$ as a bimodule between the trivial monoids on $x$ and $y$:
   \begin{equation}\label{eq:bimod_embedding}
      \fig{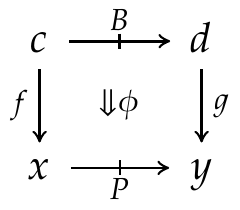}
   \end{equation}
   We will sometimes write such a bimodule embedding as ${}_f\phi_g\colon{}_M B_N\to P$, and we will
   write ${}_f\Emb_g(B,P)$ for the set of all such embeddings, for fixed embeddings $f\colon M\to x$
   and $g\colon N\to y$.

   Say an embedding~\eqref{eq:bimod_embedding} is \emph{cartesian} if $\vec{f}$, $\vec{g}$, and
   $\phi$ are all cartesian cells.
\end{definition}

\begin{example}
   A monoid $R\colon a\tickar a$ in $\dRel$ is precisely a reflexive transitive relation on the set
   $a$. An embedding $(f,\vec{f})\colon R\to x$ is a commutative diagram
   \begin{equation*}
      \fig{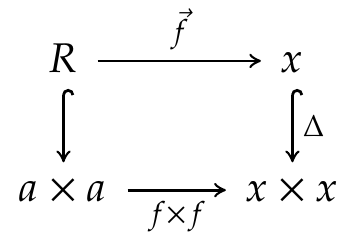}
   \end{equation*}
   or equivalently, a ``fork''
   \begin{equation*}
      \fig{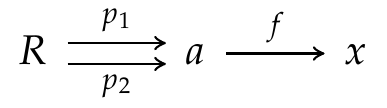}
   \end{equation*}
   i.e.~a function $f\colon a\to x$ such that $fp_1=fp_2 \; (=\vec{f})$.
\end{example}

We leave the proof of the following easy observation to the reader.

\begin{lemma}\label{lem:restriction_bimodule}
   Given embeddings $f\colon M\to x$ and $g\colon N\to y$ in a virtual equipment $\dcat{D}$, and a
   cartesian cell
   \begin{equation*}
      \fig{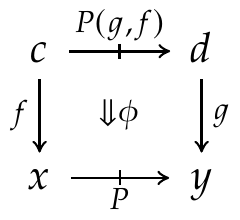}
   \end{equation*}
   there is a unique $(M,N)$-bimodule structure on $P(g,f)$ making $\phi$ an embedding.

   For any $B\in{}_M\Bimod_N$ and any proarrow $P\in\HHor(\dcat{D})(x,y)$, this construction induces
   a bijection ${}_f\Emb_g(B,P)\iso{}_M\Bimod_N(B,P(g,f))$, which is natural in $B$ and $P$.
\end{lemma}

\begin{definition}
   Lemma~\ref{lem:restriction_bimodule} determines a functor
   \begin{equation*}
      {}_f\Res_g\colon\HHor(\dcat{D})(x,y)\to{}_M\Bimod_N
   \end{equation*}
   for any pair of embeddings $f\colon M\to x$ and $g\colon N\to y$. Thus for any $P\colon x\tickar
   y$, ${}_f\Res_g(P)$ is defined to be $P(g,f)$ with the unique $(M,N)$-bimodule structure making
   the cartesian cell an embedding.
\end{definition}

\begin{definition}\label{def:monoid_collapse}
   Let $M\colon c\tickar c$ be a monoid in a virtual equipment $\dcat{D}$. A \emph{collapse} of $M$
   is a universal embedding of $M$. That is, a collapse of $M$ is an object $\Col{M}$ together with
   an embedding
   \begin{equation}\label{eq:collapse_diagram}
      \fig{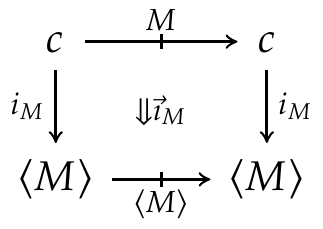}
   \end{equation}
   such that any other embedding factors uniquely through $\vec{\imath}_M$:
   \begin{equation*}
      \vcenter{\hbox{\includegraphics{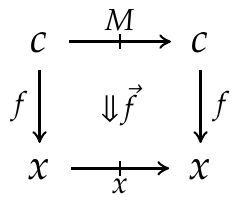}}}
      \vcenter{\hbox{\hspace{1.5em}=\hspace{1.5em}}}
      \vcenter{\hbox{\includegraphics{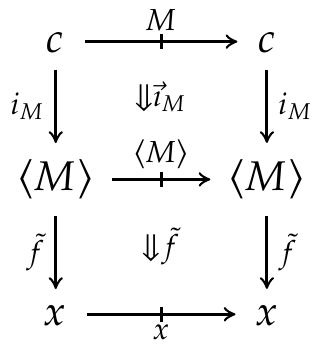}}}
   \end{equation*}
   In other words, $\Col{M}$ represents the functor $\dcat{D}_0\to\Set$ sending $x$ to $\Emb(M,x)$.
\end{definition}

\begin{example}
   Given a monoid $R\colon a\tickar a$ in $\dRel$, i.e.~a (reflexive, transitive) relation on a, the
   collapse of $R$ is a universal fork $R\rightrightarrows a\to x$, that is, a coequalizer of
   $(p_1,p_2)$.
\end{example}

\begin{example}\label{ex:collapse_in_prof}
   A monoid $M\colon C\tickar C$ in $\dProf$ is a profunctor $M\colon\op{C}\times C\to \Set$ with a
   unit and multiplication. The unit amounts to a function from morphisms $f\colon c\to d$ in $C$ to
   elements $e(f)\in M(c,d)$ of $M$, which is compatible with the functorial action on $M$ in that
   $h\cdot e(f)=e(h\circ f)$ and $e(f)\cdot g=e(f\circ g)$ whenever these make sense.

   The multiplication is an operation which, given elements $m_1\in M(c,d)$ and $m_2\in M(d,e)$,
   assigns an element $m_2\bullet m_1\in M(c,e)$. This operation must be compatible with the
   functorial action, meaning $(f\cdot m_2)\bullet m_1=f\cdot(m_2\bullet m_1)$, $(m_2\cdot g)\bullet
   m_1=m_2\bullet(g\cdot m_1)$, and $m_2\bullet(m_1\cdot h)=(m_2\bullet m_1)\cdot h$, whenever these
   make sense, and it must satisfy unit and associativity axioms: $e(f)\bullet m=f\cdot m$,
   $m\bullet e(g)=m\cdot g$, and $(m_3\bullet m_2)\bullet m_1=m_3\bullet(m_2\bullet m_1)$.

   The collapse of a monoid $M$ is a category $\Col{M}$ with the objects of $C$, and with hom sets
   $\Hom_{\Col{M}}(c,d)=M(c,d)$. The multiplication of $M$ defines the composition of $\Col{M}$,
   while the functorial action of $C$ on $M$ defines the identity-on-objects functor $i_M\colon
   C\to\Col{M}$.
\end{example}

\begin{example}\label{ex:collapse_in_ModD}
   More generally (see Example~\ref{ex:Prof_is_ModSpan}), we can form the collapse of any monoid in
   $\dMod(\dcat{D})$, for any virtual equipment $\dcat{D}$. We will sketch how this works,
   leaving the routine verifications to the reader.

   Suppose $M\colon c\tickar c$ is a monoid in $\dcat{D}$, i.e.~an object
   $(c,M)\in\dMod(\dcat{D})$, and let $N\colon(c,M)\tickar(c,M)$ be a monoid in $\dMod(\dcat{D})$.
   This means $N$ is a $(M,M)$-bimodule, together with unit and multiplication bimodule
   homomorphisms.

   The collapse of $N$ will be $N$ itself, forgetting the bimodule structure, but remembering the
   monoid structure. In particular, the unit of the collapse is the composition of the unit
   $e_M\colon c\to M$ of $M$ and the unit $\eta_N\colon M\to N$ of $N$, while the multiplication of
   the collapse is simply the multiplication of $N$.

   The collapse map $i\colon (c,M)\to(c,N)$ is the unit $\eta_N\colon M\to N$, and $\vec{\imath}$ is
   the identity on $N$.
\end{example}

\begin{definition}\label{def:bimodule_collapse}
   Let $M\colon c\tickar c$ and $N\colon d\tickar d$ be monoids in a virtual equipment such that the
   collapses $\Col{M}$ and $\Col{N}$ exist, and let $B\colon c\tickar d$ be a $(M,N)$-bimodule. A
   \emph{collapse} of $B$ is a universal embedding of $B$: a proarrow
   $\Col{B}\colon\Col{M}\tickar\Col{N}$ together with an embedding
   \begin{equation}\label{eq:bimodule_collapse}
      \fig{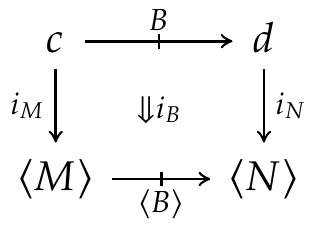}
   \end{equation}
   such that any other embedding factors uniquely through $i_B$:
   \begin{equation}
      \vcenter{\hbox{\includegraphics{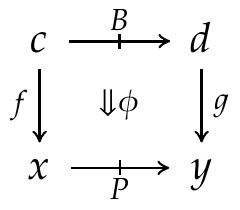}}}
      \vcenter{\hbox{\quad = \quad}}
      \vcenter{\hbox{\includegraphics{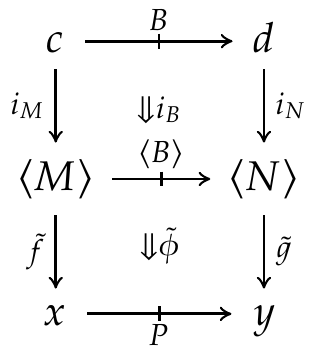}}}
   \end{equation}

   In other words, composition with $i_B$ induces a bijection
   \begin{equation*}
      {}_{\tilde{f}}\dcat{D}_{\tilde{g}}(\Col{B};\textrm{--})\iso{}_{\tilde{f}i_M}\Emb_{\tilde{g}i_N}(B,\textrm{--}).
   \end{equation*}
\end{definition}

\begin{remark}
   When it is not clear from context, we will speak of a ``monoid collapse'' or a ``bimodule
   collapse'' to specify which of Definitions~\ref{def:monoid_collapse} or
   \ref{def:bimodule_collapse} is meant.
\end{remark}

\begin{proposition}\label{prop:bimodule_collapse_characterization}
   Consider a cell in a virtual double category
   \begin{equation*}
      \fig{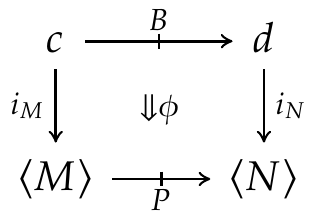}
   \end{equation*}
   where $B$ is a $(M,N)$-bimodule, $i_M\colon M\to\Col{M}$ and $i_N\colon N\to\Col{N}$ are
   collapse embeddings, and ${}_{i_M}\phi_{i_N}\colon{}_MB_N\to P$ is a bimodule embedding. The
   following are equivalent:
   \begin{compactenum}
      \item The embedding $\phi$ is a bimodule collapse.
      \item Composition with $\phi$ induces a bijection $\dcat{D}(P;\textrm{--})\iso
         {}_{i_M}\Emb_{i_N}(B,\textrm{--})$
      \item Composition with $\phi$ induces a bijection $\dcat{D}(P;\textrm{--})\iso
         {}_M\Bimod_N(B,{}_{i_M}\Res_{i_N}(\textrm{--}))$.
   \end{compactenum}
\end{proposition}
\begin{proof}
   2 and 3 are clearly equivalent by Lemma~\ref{lem:restriction_bimodule}, and 2 easily follows
   from 1.

   To see $2\Rightarrow 1$, we have the chain of equivalences, for any $f\colon\Col{M}\to x$,
   $g\colon\Col{N}\to y$, and $Q\colon x\tickar y$,
   \begin{equation*}
      {}_f\dcat{D}_g(P;Q) \iso \dcat{D}(P;Q(g,f)) \iso {}_{i_M}\Emb_{i_N}(B,Q(g,f))
         \iso{}_{fi_M}\Emb_{gi_N}(B,Q)
   \end{equation*}
\end{proof}

\begin{definition}\label{def:normal_collapse}
   We will call the diagram \eqref{eq:collapse_diagram} a \emph{normal collapse} if it also
   exhibits $\Col{M}$ as the bimodule collapse of the unit $(M,M)$-bimodule $M\colon c\tickar c$.
\end{definition}

\chapter{Regular virtual double categories}\label{ch:regular}

Let $\cat{C}$ be a category with finite limits, and let $f\colon c\to d$ be a morphism in $\cat{C}$.
Recall the following standard definitions (see e.g.~\cite{Borceux:1994b}, \cite{Bourn.Gran:2004a}):
\begin{compactitem}
   \item The \emph{kernel pair} of $f$ is the pair $p_1,p_2\colon R\rightrightarrows c$ given by the
      pullback
      \begin{equation*}
         \fig{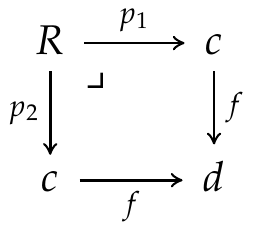}
      \end{equation*}
      A kernel pair is always an internal equivalence relation: that is $(p_1,p_2)\colon R\to
      c\times c$ is a monomorphism ($R$ is a relation), there exists a common section $c\to R$ of
      $p_1$ and $p_2$ ($R$ is reflexive), and $R$ is similarly transitive and symmetric.
   \item An equivalence relation $p_1,p_2\colon R\rightrightarrows c$ is called \emph{effective} if
      it is the kernel pair of some morphism.
   \item $f$ is a \emph{regular epimorphism} if it is the coequalizer of some parallel pair of
      arrows.
   \item $\cat{C}$ is a \emph{regular category} if every effective equivalence relation has a
      coequalizer, and if regular epimorphisms are stable under pullback.
\end{compactitem}

In a regular category $\cat{C}$, any morphism factors uniquely as a regular epimorphism followed by
a monomorphism. In fact, a category is regular if and only if it has a such a factorization system
and the regular epimorphisms are stable under pullback. In that way, regular categories are
precisely those with ``well-behaved'' image factorizations.

Another common description of regular categories is that they are precisely those with a ``good''
theory of internal relations. In particular, we have the following construction.

\begin{definition}
   For any regular category $\cat{C}$, we can define an equipment $\dRel(\cat{C})$ as follows:
   \begin{compactitem}
      \item The vertical category of $\dRel(\cat{C})$ is $\cat{C}$.
      \item Proarrows $R\colon a\tickar b$ are relations, i.e.~monomorphisms $R\hookrightarrow b\times
         a$.
      \item 2-cells
         \begin{equation}\label{eq:rel_2-cella}
            \fig{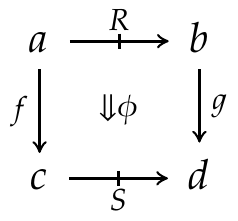}
         \end{equation}
         are commutative diagrams
         \begin{equation}\label{eq:rel_2-cellb}
            \fig{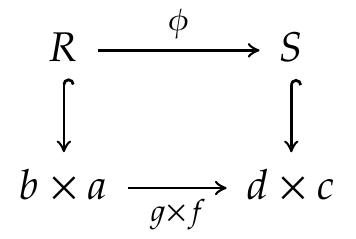}
         \end{equation}
         In particular, note that for any square of shape~\eqref{eq:rel_2-cella}, there is at most
         one 2-cell $\phi$ of that shape. We say that $\dRel(\cat{C})$ is ``locally posetal''.

         The 2-cell~\eqref{eq:rel_2-cella} is cartesian if and only if~\eqref{eq:rel_2-cellb} is a
         pullback.
      \item The unit relation $a\tickar a$ is the diagonal $\Delta\colon a\hookrightarrow a\times
         a$. The composition $R\odot S$ of two relations $R\colon a\tickar b$ and $S\colon b\tickar
         c$ is formed using pullbacks and the epi-mono factorization, as follows:
         \begin{equation*}
            \fig{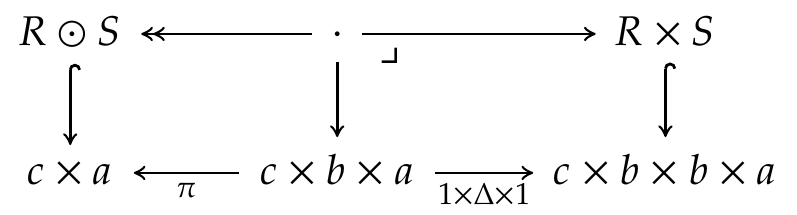}
         \end{equation*}
   \end{compactitem}
\end{definition}

In this section, we will propose a definition of \emph{regular virtual equipment}.
In~\ref{sec:regular_equipments_def} we begin with the definition and some preliminary results and
examples, and in~\ref{sec:factorization} we give a generalization of the epi-mono factorization
present in any regular category.

\section{Definition and basic properties}\label{sec:regular_equipments_def}

\begin{definition}
   Let $f\colon c\to d$ be a vertical arrow in a virtual equipment.  The \emph{kernel} of $f$ is
   defined to be the monoid obtained by restricting the trivial monoid on $d$:
   \begin{equation}\label{eq:kernel_diagram}
      \fig{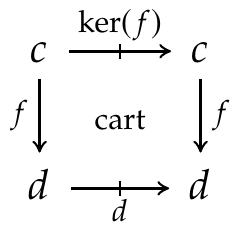}
   \end{equation}
   So the multiplication $\mu$ is the unique 2-cell satisfying
   \begin{equation}
      \vcenter{\hbox{\includegraphics{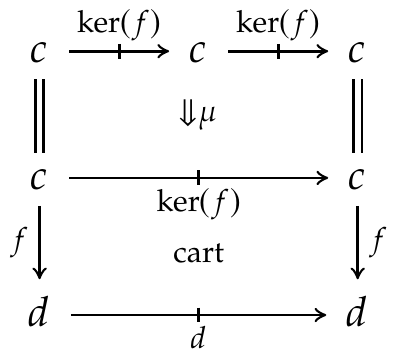}}}
      \vcenter{\hbox{\quad = \quad}}
      \vcenter{\hbox{\includegraphics{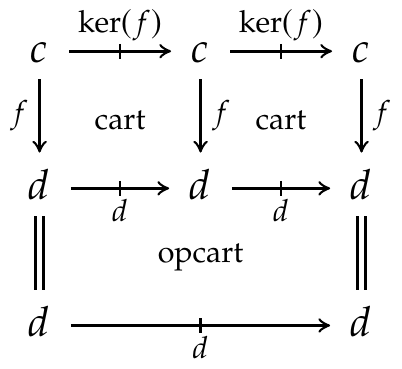}}}
   \end{equation}
   and similarly for the unit $\eta$.
\end{definition}

\begin{definition}
   Let $f\colon c\to d$ be a vertical arrow in a virtual equipment.  Say that $f$ is an
   \emph{inclusion} if the unit 2-cell on $f$ is cartesian, or equivalently if $\ker(f)$ is the
   trivial monoid on $c$. We will denote inclusions by $f\colon c\rightarrowtail d$.
\end{definition}

\begin{definition}\label{def:effective}
   Say that a monoid $M\colon c\tickar c$ in a virtual equipment is \emph{effective} if $M$ is the
   kernel of some vertical arrow.

   Similarly, say that a $(M,N)$-bimodule $B\colon c\tickar d$ is
   effective if $M$ and $N$ are effective, with $M\iso\ker(f)$ and $N\iso\ker(g)$ for some $f\colon
   c\to c'$ and $g\colon d\to d'$, and there exists a proarrow $P\colon c'\tickar d'$ such that
   $B\iso{}_f\Res_g(P)$. Equivalently, $B$ is effective if there exists a cartesian embedding
   $B\to P$ for some $P$.
\end{definition}

\begin{definition}
   Let $f\colon c\to d$ be a vertical arrow in a virtual equipment.  Say that $f$ is a
   \emph{regular cover} if the restriction~\eqref{eq:kernel_diagram} is a normal collapse cell.
   We will denote regular covers by $f\colon c\twoheadrightarrow d$.
\end{definition}

\begin{example}
   The inclusions in $\dRel(\cat{C})$ are precisely the monomorphisms of $\cat{C}$, and the regular
   covers are the regular epimorphisms.

   The inclusions in $\dProf$ are the fully-faithful functors, and the regular covers are those
   functors which are bijective on objects.
\end{example}

\begin{definition}\label{def:regular_equipment}
   Say that a virtual equipment $\dcat{D}$ is \emph{regular} if
   \begin{compactenum}
      \item \label{def:regular-ker_has_collapse}
         every effective monoid has a normal collapse,
      \item \label{def:regular-bimodule_collapse}
         for every proarrow $B\colon d\tickar d'$ and regular covers $f\colon
         c\twoheadrightarrow d$ and $g\colon c'\twoheadrightarrow d'$, the cartesian embedding
         \begin{equation*}
            \fig{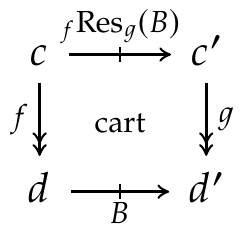}
         \end{equation*}
         is a bimodule collapse cell.
   \end{compactenum}
\end{definition}

The equipments $\dRel$ and $\dProf$ are both regular. In fact, the next two propositions show that
most ``Rel-like'' and ``Prof-like'' (virtual) equipments will be regular.
(See~\cite{Cruttwell.Shulman:2010a} for an exhibition of some of the many examples of familiar
structures arising as $\dMod(\dcat{D})$ for some $\dcat{D}$.)

\begin{proposition}\label{prop:Rel_is_regular}
   For any regular category $\cat{C}$, the equipment $\dRel(\cat{C})$ is regular.
\end{proposition}
\begin{proof}
   The kernel of any vertical morphism $f\colon a\tickar b$ is precisely the kernel pair of $f$,
   considered as an internal reflexive transitive relation $\ker(f)\colon a\tickar a$. The collapse
   of $\ker(f)$ exists because $\cat{C}$ has coequalizers of kernel pairs.

   It is not hard to check that a 2-cell~\eqref{eq:rel_2-cellb} is a bimodule collapse if and only
   if $f$, $g$, and $\phi$ are all regular epimorphisms (hint: use the orthogonality of monos and
   regular epis).

   Suppose $R\colon a\tickar a$ is an effective monoid/relation, with collapse/coequalizer $i\colon
   a\to\Col{R}$. Then in the collapse cell
   \begin{equation*}
      \fig{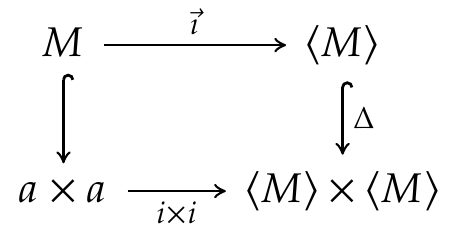}
   \end{equation*}
   we can see that $\vec{\imath}$ is a regular epimorphism as follows: $p_1$ and $p_2$ are split
   epis since $R$ is reflixive, and in a regular category every split epi is a regular epi; $i$
   is a regular epi by definition; and $\vec{\imath}=ip_1\;(=ip_2)$ is a regular epi because regular
   epis are closed under composition. Hence the collapse is normal.

   Finally, part~\ref{def:regular-bimodule_collapse} of Definition~\ref{def:regular_equipment}
   follows because regular epis are closed under product and pullback.
\end{proof}

\begin{proposition}\label{prop:ModD_is_regular}
   For any virtual equipment $\dcat{D}$, the virtual equipment $\dMod(\dcat{D})$ is regular.
\end{proposition}
\begin{proof}
   We will provide a sketch, leaving the many straightforward but tedious verifications to the
   reader.

   We saw in Example~\ref{ex:collapse_in_ModD} that in fact \emph{every} monoid $N\colon
   (c,M)\tickar (c,M)$ in $\dMod(\dcat{D})$ has a collapse. It is not hard to see that
   $\vec{\imath}_M$ is a cartesian cell, as its underlying cell in $\dcat{D}$ is the identity on
   $N$. We will see in Lemma~\ref{lem:collapse_normal_cartesian} that $\vec{\imath}_M$ being
   cartesian implies that the collapse is normal. Thus axiom 1 holds.

   To verify axiom 2, we claim that a vertical morphism $(f,\vec{f})\colon (c,M)\to (d,N)$ is a regular cover if and
   only if $f$ is an isomorphism, and that for any regular covers $f$ and $g$, a 2-cell
   $\phi\in{}_f\dMod(\dcat{D})_g(B,B')$ is a bimodule collapse if and only if it is cartesian, if
   and only if the underlying 2-cell in $\dcat{D}$ is an isomorphism.
\end{proof}

\begin{proposition}\label{prop:regular_alt_axiom}
   Condition~\ref{def:regular-bimodule_collapse} of Definition~\ref{def:regular_equipment} is
   equivalent to the following:
   \begin{compactenum}[1'.]
      \setcounter{enumi}{1}
      \item for every pair of regular covers $f\colon c\twoheadrightarrow d$ and $g\colon
         c'\twoheadrightarrow d'$, the functor
         \begin{equation*}
            {}_f\Res_g\colon\HHor(\dcat{D})(d,d')\to{}_{\ker(f)}\Bimod_{\ker(g)}
         \end{equation*}
         is fully-faithful.
   \end{compactenum}
\end{proposition}
\begin{proof}
   Let $B\colon d\tickar d'$ be a proarrow. By Proposition~\ref{prop:bimodule_collapse_characterization},
   the embedding ${}_f\Res_g(B)\to B$ is a collapse if and only if the function
   $\dcat{D}(B;\textrm{--})\to{}_{\ker(f)}\Bimod_{\ker(g)}({}_f\Res_g(B),{}_f\Res_g(\textrm{--}))$
   is a bijection.
\end{proof}

\begin{lemma}\label{lem:collapse_normal_cartesian}
   Let $\dcat{D}$ be a virtual equipment satisfying condition~\ref{def:regular-bimodule_collapse} of
   Definition~\ref{def:regular_equipment}. A collapse cell
   \begin{equation*}
      \fig{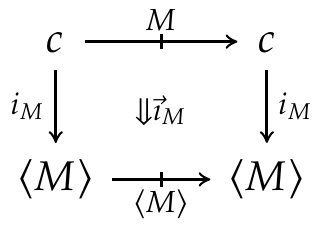}
   \end{equation*}
   in $\dcat{D}$ is normal if and only if $\vec{\imath}$ is cartesian.
\end{lemma}
\begin{proof}
   Consider the diagram (letting $i\coloneq i_M$)
   \begin{equation*}
      \fig{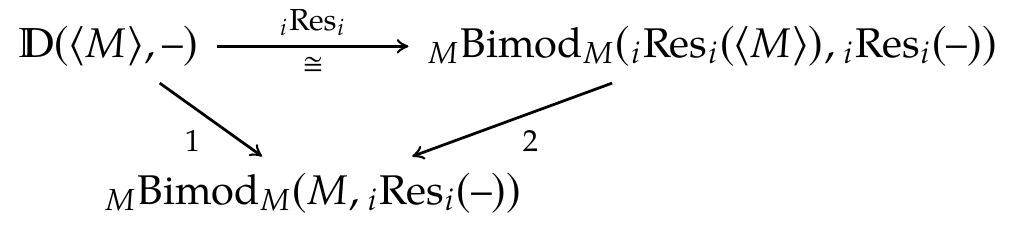}
   \end{equation*}
   in which the two downwards functions are induced by composition with $\vec{\imath}$. The top
   function is a bijection by Proposition~\ref{prop:regular_alt_axiom}. By
   Proposition~\ref{prop:bimodule_collapse_characterization}, the collapse is normal precisely if 1
   is a bijection. On the other hand, 2 is a bijection if and only if $\vec{\imath}$ induces an
   isomorphism of bimodules $M\iso{}_i\Res_i(\Col{M})$, which by
   Lemma~\ref{lem:restriction_bimodule} happens if and only if $\vec{\imath}$ is cartesian. Thus
   the collapse is normal if and only if $\vec{\imath}$ is cartesian.
\end{proof}

\begin{corollary}\label{cor:collapse_is_cartesian}
   In the presence of \ref{def:regular-bimodule_collapse},
   condition~\ref{def:regular-ker_has_collapse} of Definition~\ref{def:regular_equipment} is
   equivalent to:
   \begin{compactenum}[1'.]
      \item Every effective monoid $M$ has a collapse $(i,\vec{\imath})$, and $\vec{\imath}$ is
         cartesian.
   \end{compactenum}
   In other words, in a regular virtual equipment, every kernel is the kernel of its collapse.
\end{corollary}

\begin{proposition}\label{prop:effective_bimod_has_collapse}
   Any effective bimodule in a regular virtual equipment has a collapse, and moreover the collapse
   cell is cartesian.
\end{proposition}
\begin{proof}
   Let $M$ and $N$ be effective monoids, and suppose given an effective $(M,N)$-bimodule $B\colon
   c\tickar c'$, with cartesian embedding
   \begin{equation}\label{eq:bimod_collapse_lemma}
      \fig{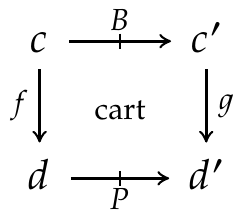}
   \end{equation}
   We can factor this as
   \begin{equation*}
      \fig{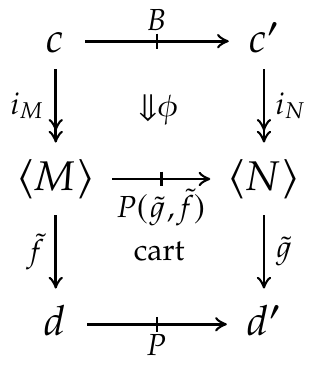}
   \end{equation*}
   It follows that $\phi$ is cartesian, hence a bimodule collapse (noting that
   $\ker(i_M)\iso M$ by Corollary~\ref{cor:collapse_is_cartesian}, and similarly for $g$).
\end{proof}

\begin{lemma}\label{lem:regular_cart_lemma}
   Consider a diagram in a regular virtual equipment $\dcat{D}$ of the form
   \begin{equation}\label{eq:cart_lemma}
      \fig{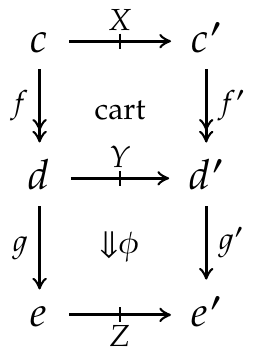}
   \end{equation}
   in which $f$ and $f'$ are regular covers. If the composite 2-cell is cartesian, then $\phi$ is
   also cartesian.
\end{lemma}
\begin{proof}
   We could also factor the composite 2-cell through the restriction $Z(g',g)$:
   \begin{equation*}
      \fig{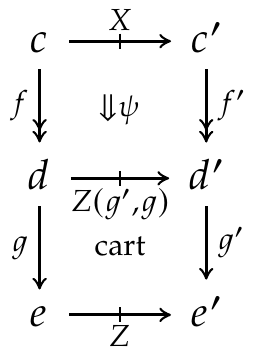}
   \end{equation*}
   If the composite is cartesian, then so is $\psi$, hence by
   part~\ref{def:regular-bimodule_collapse} of Definition~\ref{def:regular_equipment} $\psi$ is also
   a bimodule collapse cell. The upper 2-cell in~\eqref{eq:cart_lemma} is a bimodule collapse cell
   for the same reason. But then by the universal property of bimodule collapse, this factorization
   is in fact isomorphic to~\eqref{eq:cart_lemma}, hence $\phi$ is cartesian.
\end{proof}

\section{The factorization system}\label{sec:factorization}

One of the primary facts about any regular category is the existence of an image factorization.
In a regular category $\cat{C}$ there is an orthogonal factorization system
$(\mathscr{E},\mathscr{M})$ where $\mathscr{E}$ is the class of regular epimorphisms and
$\mathscr{M}$ is the class of monomorphisms. We will now see that a regular virtual equipment admits
an analogous orthogonal factorization system.

In a regular category, the image of a morphism is defined to be the coequalizer of its kernel, and
it is shown that any morphism factors through its image. We can perform the analogous construction
in a regular virtual equipment: for any vertical arrow $f\colon c\to d$ we define its image to be
the collapse $\Col{\ker(f)}$ of its kernel, and we get a unique arrow
$\tilde{f}\colon\Col{\ker(f)}\to d$ such that
\begin{equation}\label{eq:image_fact}
   \vcenter{\hbox{\includegraphics{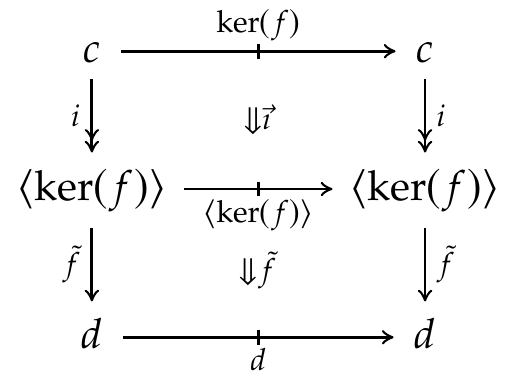}}}
   \vcenter{\hbox{\quad = \quad}}
   \vcenter{\hbox{\includegraphics{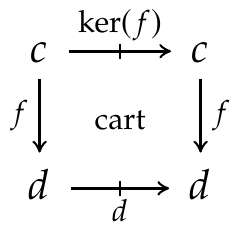}}}
\end{equation}

Another standard fact from the theory of regular categories is that the classes of regular
epimorphisms and strong epimorphisms coincide, where a morphism $f$ is called a strong epimorphism
if it is left-orthogonal to the class of monomorphisms. We begin with an analogous definition in the
setting of virtual equipments.

\begin{definition}
   Let $f\colon a\to b$ be a vertical arrow in a virtual equipment $\dcat{D}$. Say that $f$ is a
   \emph{strong cover} if it is left 2-orthogonal to the class of inclusions in the vertical 2-category
   $\VVer(\dcat{D})$, i.e.~if for any inclusion $g\colon c\rightarrowtail d$ the commuting square
   \begin{equation*}
      \fig{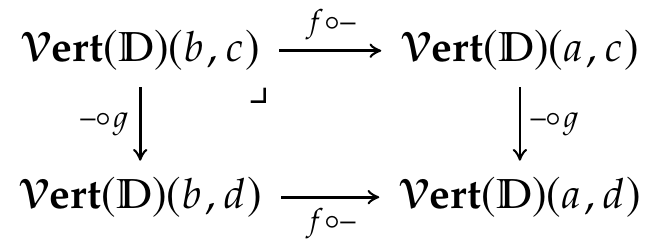}
   \end{equation*}
   is a (strict) pullback of categories.
\end{definition}

\begin{proposition}\label{prop:regular_covers_are_strong}
   Any regular cover $f\colon a\twoheadrightarrow b$ in a virtual equipment $\dcat{D}$ is a strong
   cover.
\end{proposition}
\begin{proof}
   Suppose we have an inclusion $g\colon c\rightarrowtail d$ and a commutative square $v\circ f =
   g\circ u$ in $\VVer(\dcat{D})$. We need to show there is a unique arrow $h\colon b\to c$ such
   that $g\circ h=v$ and $h\circ f=u$. Because $g$ is an inclusion, there is a unique 2-cell $\phi$
   satisfying
   \begin{equation}\label{eq:orthogA}
      \fig{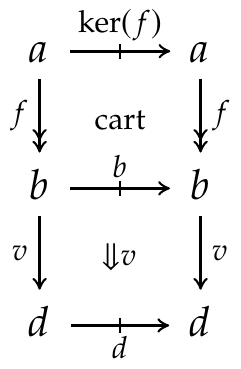}
      \vcenter{\hbox{\quad = \quad}}
      \fig{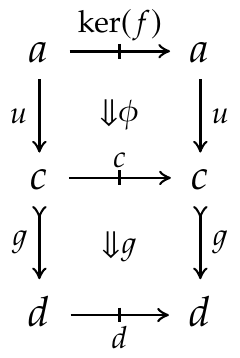}
   \end{equation}
   and it is not hard to check, again using that $g$ is an inclusion, that $(u,\phi)$ is an embedding
   $\ker(f)\to c$.  Then, because $f$ is a regular cover, there is a unique arrow
   $h\colon b\to c$ satisfying
   \begin{equation}\label{eq:orthogB}
      \fig{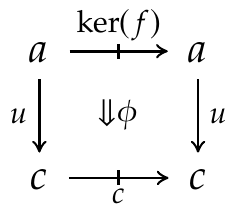}
      \vcenter{\hbox{\quad = \quad}}
      \fig{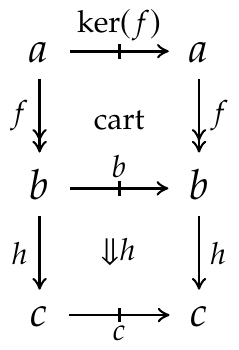}
   \end{equation}
   We can read $h\circ f=u$ directly off \eqref{eq:orthogB}, while $g\circ h=v$ follows because it
   becomes true after precomposition with $f$.

   If $h'$ is another arrow such that $h'\circ f=u$ and $g\circ h'=v$, then \eqref{eq:orthogB}
   with $h'$ in place of $h$ holds because it becomes true after postcomposition with $g$, and
   therefore $h'=h$ by the universality of $f$.

   For the 2-dimensional orthogonality, suppose we have 2-cells $\alpha\colon u\Rightarrow u'$ in
   $\VVer(\dcat{D})(a,c)$ and $\beta\colon v\Rightarrow v'$ in $\VVer(\dcat{D})(b,d)$, such that
   $g\circ\alpha=\beta\circ f$. Similarly to \eqref{eq:orthogA}, there is a unique 2-cell $\psi$
   satisfying
   \begin{equation}\label{eq:orthogC}
      \fig{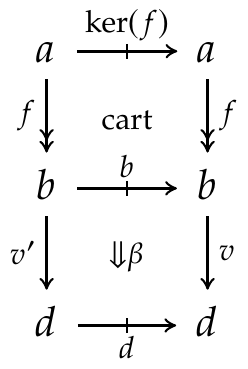}
      \vcenter{\hbox{\quad = \quad}}
      \fig{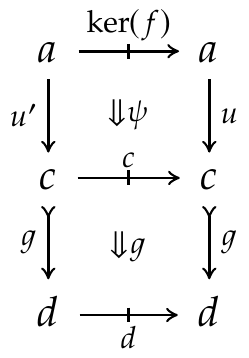}
   \end{equation}
   and, using once more that $g$ is an inclusion, one can verify that $\psi$ is a bimodule embedding
   $\tensor[_{u'}]{\psi}{_{u}}\colon\tensor[_{\ker(f)}]{\ker(f)}{_{\ker(f)}}\to c$, and also that
   $\psi\circ e_{\ker(f)}=\alpha$, where $e_{\ker(f)}\colon a\Rightarrow \ker(f)$ is the unit of the
   monoid $\ker(f)$. Because $f$ is a regular cover, hence $\ker(f)\to b$ is a bimodule collapse,
   there is a unique $\gamma$ such that
   \begin{equation}\label{eq:orthogD}
      \fig{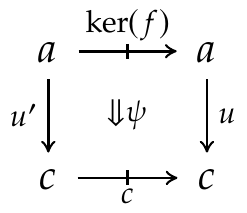}
      \vcenter{\hbox{\quad = \quad}}
      \fig{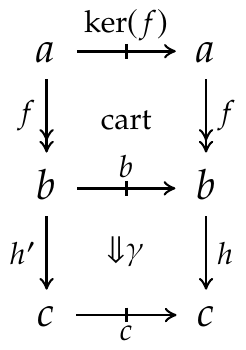}
   \end{equation}
   and this $\gamma$ is the 2-cell $h\Rightarrow h'$ in $\VVer(\dcat{D})(b,c)$ we wanted. By
   precomposing \eqref{eq:orthogD} with the unit $e_{\ker(f)}$ of the monoid $\ker(f)$, we get
   $\gamma\circ f = \alpha$, and $g\circ\gamma=\beta$ holds because it becomes true after precomposing
   with the collapse $\ker(f)\to b$.

   Finally, verifying that another $\gamma'$ satisfying $\gamma'\circ f=\alpha$ and
   $g\circ\gamma'=\beta$ must be equal to $\gamma$ is analogous to the uniqueness of $h$ above.
\end{proof}

\begin{theorem}\label{thm:image_factorization}
   Let $\dcat{D}$ be a regular virtual equipment. There is an orthogonal
   2-factorization system $(\mathcal{E},\mathcal{M})$ on the vertical 2-category $\VVer(\dcat{D})$,
   where $\mathcal{E}$ is the class of regular covers, and $\mathcal{M}$ is the class of
   inclusions.
\end{theorem}
\begin{proof}
   The orthogonality of these two classes was proven in
   Proposition~\ref{prop:regular_covers_are_strong}. The factorization is constructed as
   in~\eqref{eq:image_fact}. The arrow $i\colon c\twoheadrightarrow\Col{\ker(f)}$ is clearly
   regular, and that $\tilde{f}$ is an inclusion follows directly from
   Lemma~\ref{lem:regular_cart_lemma}.
\end{proof}

\begin{corollary}
   In a regular virtual equipment, the classes of strong covers and regular covers coincide.
\end{corollary}
\begin{proof}
   By Proposition~\ref{prop:regular_covers_are_strong} we know that every regular cover is a strong
   cover.

   Given a strong cover $f\colon c\twoheadrightarrow d$, by Theorem~\ref{thm:image_factorization} we can
   factor $f=\tilde{f}\circ i$ with $i$ regular cover and $\tilde{f}$ an inclusion. Because
   $\tilde{f}\circ i$ is a strong cover, it follows that $i$ is a strong cover as well, hence an
   isomorphism. Thus $f$ is a regular cover because $\tilde{f}$ is.
\end{proof}

\chapter{Exact virtual double categories}\label{ch:exact}

Recall that a category $\cat{C}$ with finite limits is called \emph{exact} if every internal
equivalence relation in $\cat{C}$ is effective.

\begin{definition}
   Let $\dcat{D}$ be a virtual equipment. Say that $\dcat{D}$ is \emph{exact} if $\dcat{D}$ is
   regular, and if every monoid and bimodule in $\dcat{D}$ is effective.
   (See~\ref{def:regular_equipment}, \ref{def:effective}.)
\end{definition}

\begin{proposition}
   For any virtual equipment $\dcat{D}$, the virtual equipment $\dMod(\dcat{D})$ is exact.
\end{proposition}
\begin{proof}
   We saw in Proposition~\ref{prop:ModD_is_regular} that $\dMod(\dcat{D})$ is regular. Additionally,
   we saw in Example~\ref{ex:collapse_in_ModD} that in fact \emph{all} monoids in $\dMod(\dcat{D})$
   have a collapse, and it is clear from the construction that the collapse cell is cartesian. Hence
   every monoid in $\dMod(\dcat{D})$ is effective.

   From the proof of Proposition~\ref{prop:ModD_is_regular}, it is clear that any bimodule in
   $\dMod(\dcat{D})$ has a collapse with the same underlying proarrow in $\dcat{D}$, and that the
   collapse cell is cartesian. Hence every bimodule in $\dMod(\dcat{D})$ is effective.
\end{proof}

\begin{remark}
   We might hope to say that for any exact category $\cat{C}$, the virtual equipment
   $\dRel(\cat{C})$ is exact, extending Proposition~\ref{prop:Rel_is_regular}. However, this is not
   the case. For $\dRel(\cat{C})$ to be exact would mean that every reflexive and transitive
   relation (not necessarily symmetric) is the kernal pair of some morphism. This would imply that
   every reflexive and transitive relation \emph{is} symmetric, and this is clearly not true in
   general.

   It appears that exactness for a virtual equipment is a ``directed'' generalization of exactness
   for a category. This directedness is essential to the category-like examples, where the elements
   of a monoid $M$ become the morphisms in its collapse $\Col{M}$. Moreover, it is not even possible
   to \emph{define} what a symmetric monoid in a virtual equipment is without some extra structure.
\end{remark}

\begin{proposition}\label{prop:exact_characterization}
   A virtual equipment $\dcat{D}$ is exact if and only if:
   \begin{compactitem}
      \item every monoid $M\colon c\tickar c$ has a collapse $(i,\vec{\imath})\colon (c,M)\to\Col{M}$
         with $\vec{\imath}_M$ cartesian, and
      \item for every pair of monoids $M,N$, the restriction functor
         \begin{equation*}
            {}_{i_M}\Res_{i_N}\colon\HHor(\dcat{D})(\Col{M},\Col{N})\to{}_M\Bimod_N
         \end{equation*}
         is an equivalence of categories.
   \end{compactitem}
\end{proposition}
\begin{proof}
   To begin, suppose $\dcat{D}$ satisfies the conditions of the proposition. Clearly, this implies
   that every monoid and bimodule in $\dcat{D}$ is effective, and that every effective monoid has a
   collapse. The only thing remaining to check is part~\ref{def:regular-bimodule_collapse} of
   Definition~\ref{def:regular_equipment}.

   Suppose we have a cartesian cell of the form
   \begin{equation*}
      \fig{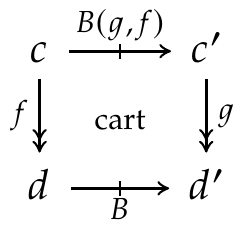}
   \end{equation*}
   Let $M=\ker(f)$ and $N=\ker(g)$, and without loss of generality let $d=\Col{M}$ and $f=i_M$, and
   similarly for $d'$ and $g$. For any vertical arrows $h\colon\Col{M}\to x$ and $h'\colon\Col{N}\to
   x'$ and any proarrow $P\colon x\tickar x'$, we have a string of bijections
   \begin{align*}
      {}_{hi_M}\Emb_{h'i_N}(B(i_N,i_M),P)
      & \iso {}_M\Bimod_N\left(B(i_N,i_M),P(h'i_N,hi_M)\right) \\
         & \iso \dcat{D}(B,P(h',h)) \\
         & \iso \tensor[_h]{\dcat{D}}{_{h'}}(B,P)
   \end{align*}
   where the first is by Lemma~\ref{lem:restriction_bimodule}, the second is the second condition of
   the proposition, and the third is the definition of restriction. This shows that $B$ is the
   collapse of $B(i_N,i_M)=B(g,f)$.

   Conversely, suppose $\dcat{D}$ is exact. By assumption, any monoid $M\colon c\tickar c$ is
   effective, hence $M$ has a collapse because $\dcat{D}$ is regular, and $\vec{\imath}_M$ is
   cartesian by Corollary~\ref{cor:collapse_is_cartesian}.

   For the second condition, because $\dcat{D}$ is regular we already know from
   Proposition~\ref{prop:regular_alt_axiom} that the restriction functor
   $\HHor(\dcat{D})(\Col{M},\Col{N})\to{}_M\Bimod_N$ is fully faithful. To see that it is
   essentially surjective, let $B\in{}_M\Bimod_N$ be a bimodule. Any $(M,N)$-bimodule is effective,
   so $B$ has a collapse $\Col{B}\colon\Col{M}\tickar\Col{N}$ by
   Proposition~\ref{prop:effective_bimod_has_collapse}, and moreover the embedding $\vec{\imath}_B$
   is cartesian.  Hence $B\iso{}_{i_M}\Res_{i_N}(\Col{B})$.
\end{proof}

In~\cite{Wood:1982a,Wood:1985a}, proarrow equipments are introduced as a proposed setting for formal
category theory. There the structure of a proarrow equipment was presented in terms of an
identity-on-objects pseudo 2-functor $(\textrm{--})_*\colon\ccat{K}\to\ccat{M}$ between bicategories.
In~\cite{Shulman:2008a} it is proven that an equipment (there called a framed bicategory), can be
equivalently defined to be a pseudo 2-functor $\overbar{(\textrm{--})}\colon\ccat{K}\to\ccat{M}$,
where $\ccat{K}$ is a strict 2-category and $\ccat{M}$ is a bicategory with the same objects,
$\overbar{(\textrm{--})}$ is the identity on objects and locally fullly-faithful, and such that for
every arrow $f$ in $\ccat{K}$, $\overbar{f}$ has a right adjoint $\widetilde{f}$ in $\ccat{M}$. This
is equivalent to Wood's definition, except that $\ccat{K}$ is required to be a strict 2-category.

If $\dcat{D}$ is a framed bicategory, then the corresponding proarrow equipment has
$\ccat{K}=\VVer(\dcat{D})$ the vertical 2-category and $\ccat{M}=\HHor(\dcat{D})$ the horizontal
bicategory of $\dcat{D}$, while for any vertical arrow $f\colon c\to d$, $\overbar{f}=d(1,f)\colon
\tickar d$ is the representable proarrow, which has a right adjoint $\widetilde{f}=d(f,1)$.

However, in~\cite{Wood:1985a} two more axioms are proposed to support the development of formal
category theory. The first of these concerns coproducts, which we will not be considering in this
paper. The second, there called Axiom 5, concerns Kleisli objects for monads in $\ccat{M}$.

Recall that given a monad $M\colon a\to a$ in a bicategory $\ccat{B}$, a \emph{left $M$-module} is
an arrow $X\colon a\to b$ together with an action $X\circ M\Rightarrow X$ satisfying the usual
axioms for monoid action. A homomorphism of left $M$-modules from $X\colon a\to b$ to $X'\colon
a\to b'$ is an arrow $f\colon b\to b'$ and a 2-cell $f\circ X\Rightarrow X'$ which respects the $M$
actions in the obvious way.  This defines for any monad $M$ a functor $\LMod(\textrm{--},M)$ taking
an object $b$ to the category of left $M$-modules $a\to b$. The \emph{Kleisli object} for $M$ is
then defined to be an object $a_M$ which represents this functor, i.e.~equipped with a natural
equivalence $\ccat{B}(a_M,b)\iso\LMod(b,M)$. We will refer to the left $M$-module $a\to a_M$
corresponding to the identity on $a_M$ as the \emph{universal left $M$-module}.

Dually, the \emph{Eilenberg-Moore} object, or EM object, $a^M$ of $M$ is a representing object for
the functor $\RMod(\textrm{--},M)$ sending an object $b$ to the category of right $M$-modules $b\to
a$.

If $(\textrm{--})_*\colon\ccat{K}\to\ccat{M}$ is a proarrow equipment, then Wood's Axiom 5 requires
that every monad $M\colon a\to a$ in $\ccat{M}$ has a representable Kleisli object $(i_M)_*\colon
a\to a_M$ such that the adjoint $(i_M)^*\colon a_M\to a$ is an EM object for $M$, and such that a
composition $f\circ(i_M)_*$ is representable if and only if $f$ is.

We will give a slightly strictified version of this axiom, which is appropriate when assuming
$\ccat{K}$ is a strict 2-category, and then show that this is equivalent to exactness as we defined
above. But first we will prepare with a lemma to help translate between the double category
formalism and the proarrow equipment formalism

\begin{lemma}\label{lem:framed_bicat_monoid_actions}
   Let $M\colon c\tickar c$ be a monoid in an equipment $\dcat{D}$. For any 2-cell of the
   form
   \begin{equation}\label{eq:framed_bicat_monoid_actions}
      \fig{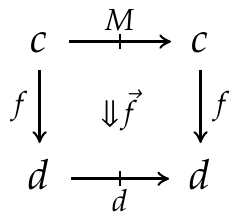}
   \end{equation}
   the following are equivalent:
   \begin{compactitem}
      \item $(f,\vec{f})$ is an embedding,
      \item the corresponding 2-cell $\vec{f}_*\colon M\odot d(1,f)\Rightarrow d(1,f)$ in
         $\HHor(\dcat{D})$ is a left $M$-action,
      \item the corresponding 2-cell $\vec{f}^*\colon d(f,1)\odot M\Rightarrow d(f,1)$ in
         $\HHor(\dcat{D})$ is a right $M$-action.
   \end{compactitem}
\end{lemma}

\begin{definition}\label{def:Woods_axiom}
   Let $\dcat{D}$ be an equipment. Say that $\dcat{D}$ \emph{satisfies Wood's axiom 5} if,
   for every monoid $M\colon c\tickar c$, there is an object $c_M$, vertical arrow $i\colon
   c\to c_M$, and 2-cell
   \begin{equation}\label{eq:Woods_axiom}
      \fig{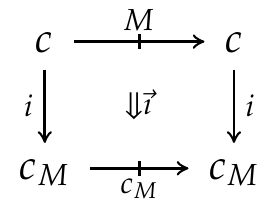}
   \end{equation}
   such that
   \begin{compactitem}
      \item the corresponding 2-cell $\vec{\imath}_*\colon
         M\odot c_M(1,i)\Rightarrow c_M(1,i)$ in $\HHor(\dcat{D})$ is a universal left
         $M$-module,
      \item the corresponding 2-cell $\vec{\imath}\,^*\colon c_M(i,1)\odot
         M\Rightarrow c_M(i,1)$ in $\HHor(\dcat{D})$ is a universal right $M$-module, and
      \item any proarrow $P\colon c_M\tickar d$ is representable if (and only if)
         $c_M(1,i)\odot P$ is. Moreover, if the latter is represented by $g\colon c\to d$, then
         $P$ is representable by some $f$ such that $f\circ i=g$ (an equality, not just an
         isomorphism).
   \end{compactitem}
\end{definition}

\begin{lemma}\label{lem:bicat_EM_Kleisli}
   For any bicategory $\ccat{B}$ the following are equivalent:
   \begin{compactenum}
      \item Every monad $M\colon c\to c$ in $\ccat{B}$ has an object $c_M$ which is both the
         Kleisli object and EM object for $M$.
      \item Every monad $M\colon c\to c$ in $\ccat{B}$ factors as an adjunction $i_M\dashv i^M$,
         $M\iso i^M\circ i_M\colon c\to c_M\to c$, such that for every pair $M,N$ of monads the
         induced functor
         \begin{equation}
            \fig{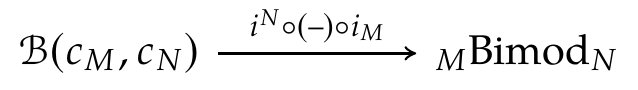}
         \end{equation}
         is an equivalence of categories.
   \end{compactenum}
\end{lemma}
\begin{proof}
   ($1\Rightarrow 2$): Let $M$ be a monad with universal right $M$-module $i^M\colon c_M\to c$
   and universal left $M$-module $i_M\colon c\to c_M$. It is a standard fact from bicategory
   theory (see e.g.~\cite{Street:1972a}) that by factoring the unit right $M$-module $M$ through the
   universal one, $M\iso i^M\circ\alpha$, we get an adjunction $\alpha\dashv i^M$ such that $M$ is
   the monad induced by the adjunction. If we similarly factor the unit left $M$-module,
   $M\iso\beta\circ i_M$, we get an adjunction $i_M\dashv \beta$. It follows that $\alpha\iso i_M$,
   $\beta\iso i^M$, and $M\iso i^M\circ i_M$.

   To see the equivalence of categories, we only need to note that composition with $i^N$ induces an
   equivalence $\LMod(c_N,M)\iso{}_N\Bimod_M$, and likewise for $i_M$. This is a straightforward
   check which we leave to the reader. Thus each functor in
   \begin{equation*}
      \fig{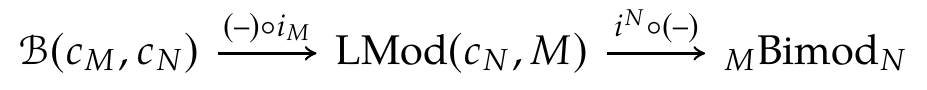}
   \end{equation*}
   is an equivalence.

   ($2\Rightarrow 1$): Let $M\colon c\to c$ be a monad. To see that $i_M$ is a universal left
   $M$-module, simply notice that $\LMod(b,M)\iso{}_M\Bimod_{1_b}$. Thus $i_M$ is a universal left
   $M$-module because
   \begin{equation*}
      \fig{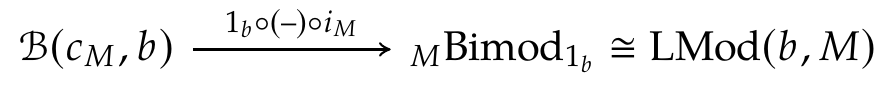}
   \end{equation*}
   is an equivalence of categories. Likewise we can see that $i^M$ is a universal right $M$-module.
\end{proof}


\begin{theorem}
   A framed bicategory $\dcat{D}$ is exact if and only if it satisfies Wood's axiom 5.
\end{theorem}
\begin{proof}
   \textbf{Axiom 5 $\Rightarrow$ exact:} We will use Proposition~\ref{prop:exact_characterization}
   to show that $\dcat{D}$ is exact.

   Let $M\colon c\tickar c$ be a monoid in $\dcat{D}$, and let $(i,\vec{\imath}\,)\colon
   M\to c_M$ be the embedding in~\eqref{eq:Woods_axiom}. We wish to show that this is a collapse
   cell, hence $c_M\iso\Col{M}$.

   Let $(f,\vec{f}\,)\colon M\to x$ be any other embedding. By
   Lemma~\ref{lem:framed_bicat_monoid_actions}, this makes $d(1,f)$ a left $M$-module, hence there
   is a unique-up-to-isomorphism proarrow $X\colon c_M\tickar d$ such that $c_M(1,i)\odot
   X\iso d(1,f)$ and $\vec{\imath}_*\odot X=\vec{f}_*$. By Definition~\ref{def:Woods_axiom}, $X$ is
   representable by a unique $\tilde{f}$ such that $\tilde{f}\circ i=f$. Finally, under the bijection
   of Proposition~\ref{prop:equipment_2-cell_bijection}, the equation $\vec{\imath}_*\odot
   X=\vec{f}_*$ becomes $\tilde{f}\circ\vec{\imath}=\vec{f}$. Thus any embedding $(f,\vec{f})$
   factors uniquely through $(i,\vec{\imath}\,)$, making $c_M$ the collapse of $M$. Moreover,
   the collapse cell $\vec{\imath}$ is cartesian, corresponding to the canonical isomorphsim
   $M\iso c_M(1,i)\odot c_M(i,1)$ from Lemma~\ref{lem:bicat_EM_Kleisli}.

   Thus we have shown the first condition of Proposition~\ref{prop:exact_characterization}, and the
   second follows directly from Lemma~\ref{lem:bicat_EM_Kleisli}, hence $\dcat{D}$ is exact.

   \textbf{Exact $\Rightarrow$ axiom 5:} For every monoid $M$ we will take the
   2-cell~\eqref{eq:Woods_axiom} to be the collapse cell of $M$. That $\Col{M}$ is both the Kleisli
   and the EM object for $M$ in $\HHor(\dcat{D})$ follows from Lemma~\ref{lem:bicat_EM_Kleisli} and
   Proposition~\ref{prop:exact_characterization}.

   For the last bullet of Definition~\ref{def:Woods_axiom}, let $X\colon\Col{M}\tickar d$ be a
   proarrow, and suppose that $\Col{M}(1,i)\odot X\iso d(1,g)$ for some $g\colon c\to d$. Then
   $d(1,g)$ is a left $M$-module, and by Lemma~\ref{lem:framed_bicat_monoid_actions} this left $M$
   action is equivalent to an embedding $(g,\vec{g})\colon M\to d$. Factoring this embedding through
   the collapse $g=\tilde{g}\circ i_M$ gives an isomorphsim $d(1,g)\iso\Col{M}(1,i)\odot
   d(1,\tilde{g})$ of left $M$-modules, and because $\Col{M}(1,i)$ is the universal left $M$-module,
   this implies $X\iso d(1,\tilde{g})$.
\end{proof}

\printbibliography
\end{document}

%% file: preamble.tex
\usepackage{mathtools}
\usepackage{amsthm}
\usepackage{amssymb}
\usepackage[cal=euler,scr=rsfso]{mathalfa}
\usepackage{bm}
\usepackage{inputenc}
\usepackage{microtype}
\usepackage[usenames,dvipsnames]{xcolor}
\usepackage[neverdecrease,neveradjust]{paralist}
\usepackage{booktabs}
\usepackage{tensor}
\usepackage{todonotes}
\usepackage[backend=bibtex8]{biblatex}

\DeclareMathOperator{\Hom}{Hom}

\theoremstyle{plain}
\newtheorem{theorem}{Theorem}[chapter]
\newtheorem*{theorem*}{Theorem}

\newtheorem{proposition}[theorem]{Proposition}
\newtheorem{corollary}[theorem]{Corollary}
\newtheorem{lemma}[theorem]{Lemma}
\newtheorem*{lemma*}{Lemma}

\newtheorem*{namedthm}{\namedthmname}
\newcounter{namedthm}


\theoremstyle{definition}
\newtheorem{definition}[theorem]{Definition}

\theoremstyle{remark}
\newtheorem{example}[theorem]{Example}
\newtheorem{remark}[theorem]{Remark}

\input{figures/fig_opts}

\newcommand{\erase}[1]{{}}

\allowdisplaybreaks

%% file: figures/fig_opts.tex
\newcommand{\iso}{\cong}

\newcommand{\cat}[1]{\mathscr{#1}} 
\newcommand{\ncat}[1]{\mathbf{#1}} 
\newcommand{\ccat}[1]{\mathcal{#1}} 
\newcommand{\dcat}[1]{\mathbb{#1}} 

\newcommand{\Set}{\ncat{Set}} 
\newcommand{\dSpan}{\dcat{S}\ncat{pan}} 
\newcommand{\Cat}{\ncat{Cat}} 
\newcommand{\CCat}{\bm{\ccat{C}}\ncat{at}} 
\newcommand{\dProf}{\dcat{P}\ncat{rof}} 

\newcommand{\dMod}{\dcat{M}\ncat{od}} 
\newcommand{\VVer}{\bm{\ccat{V}}\ncat{ert}} 
\newcommand{\HHor}{\bm{\ccat{H}}\ncat{or}} 
\newcommand{\dRel}{\dcat{R}{\ncat{el}}}

\newcommand{\Col}[1]{\langle#1\rangle}

\newcommand{\op}[1]{{#1}^{\text{op}}}

\newcommand{\RMod}{\mathrm{RMod}}
\newcommand{\LMod}{\mathrm{LMod}}
\newcommand{\Bimod}{\mathrm{Bimod}}
\newcommand{\Emb}{\mathrm{Emb}}
\newcommand{\Res}{\mathrm{Res}}

\newcommand{\overbar}[1]{\mkern 1.5mu\overline{\mkern-1.5mu#1\mkern-1.5mu}\mkern 1.5mu}

\newcommand{\tickar}{\includegraphics{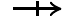}}